\documentclass{article}
\usepackage{geometry}
\usepackage{titling}
\AtEndDocument{\bigskip{\footnotesize%
\textsc{$^*$School of Computing, Dehradun Institute of Technology, Dehradun, Uttarakhand–248009, India} \par  
  \textit{E-mail address:} \texttt{suryagiri456@gmail.com} \par

}}
\usepackage{doc}

\usepackage{url}

\usepackage{graphicx}
\usepackage{epstopdf}
\usepackage{amsthm}
\usepackage{amssymb}
\usepackage{amsmath}
\usepackage{enumitem}

\usepackage{hypdoc}

\usepackage{array}
\usepackage{enumitem}
\usepackage[utf8]{inputenc}
\usepackage[english]{babel}
\newtheorem{theorem}{Theorem}
\newtheorem{corollary}{Corollary}[theorem]
\newtheorem{lemma}[theorem]{Lemma}
\newtheorem*{definition}{Definition}

\DeclareMathOperator{\RE}{Re}

\usepackage{varwidth}
\usepackage{lineno, hyperref}
\usepackage{graphicx}
\usepackage{epstopdf}
\usepackage{amsmath}
\usepackage{amsthm}
\usepackage{enumitem}
\usepackage{amsthm}
\usepackage{float}
\usepackage{subcaption}
\usepackage{caption}
\usepackage{authblk}
\usepackage{abstract}
\usepackage{thmtools}
\declaretheorem[numbered=no,
name=Theorem A]{theoremA}

\declaretheorem[numbered=no,
name=Theorem B]{theoremB}

\begin{document}
\title{Toeplitz determinant and generalized Zalcman conjecture for subclasses of starlike mappings in higher dimensions}
\author{Surya Giri$^*$}


\date{}


	

\maketitle	

\begin{abstract}
    \noindent     In  this manuscript, we establish sharp bounds of the third-order Toeplitz determinant and a particular case of the generalized Zalcman functional for a  class of holomorphic mappings defined on the unit ball in a complex Banach space. The obtained estimates yield corresponding bounds for several subclasses of starlike mappings as special cases and also provide higher-dimensional extensions of certain known results from the classical one-dimensional theory.
\end{abstract}
\vspace{0.5cm}

\section{Introduction}\label{sec1}
   Let $\mathcal{A}$ denote the class of all analytic functions of the form $f(\zeta)=\zeta+\sum_{n=2}^\infty a_n \zeta^n$ defined on the open unit disk
   $\mathbb{U}=\{ \zeta\in \mathbb{C}: \vert \zeta \vert < 1 \}.$
   By $\mathcal{S}$, we represent the subclass of $\mathcal{A}$ containing all univalent functions. Further, let
   $\mathcal{S}^*$, $\mathcal{S}^*(\alpha)$ and $\mathcal{SS}^*(\beta)$ be the subclasses of $\mathcal{S}$ consisting of starlike functions, starlike functions of order $\alpha$ ($0 \leq  \alpha <1$) and strongly starlike functions of order $\beta$ ($0 <\beta \leq 1$), respectively. For more details on these classes, we refer the reader to~\cite{Goodman}.

   Ali et al.~\cite{AliThoVas} first studied the Toeplitz determinant for certain subclasses of $\mathcal{S}$. For $f \in \mathcal{A},$ Toeplitz determinant is defined by
\begin{equation*}
     T_{m,n}(f)= \begin{vmatrix}
	a_n & a_{n+1} & \cdots & a_{n+m-1} \\
	a_{n+1} & a_n & \cdots & a_{n+m-2}\\
	\vdots & \vdots & \ddots & \vdots\\
    a_{n+m-1} & a_{n+m-2} & \cdots & a_n\\
	\end{vmatrix}.
\end{equation*}
    Consequently, we have
    $ T_{2,1}(f) = a_2^2 - a_3^2$,  $ T_{3,1}(f) = 1 - 2 a_2^2 + 2 a_2^2 a_3 -  a_3^2  $  and
\begin{equation}
      T_{3,2}(f)  =  a_2^3-2 a_2 a_3^2- a_2 a_4^2 + 2 a_3^2 a_4 .
\end{equation}
     Ali et al.~\cite{AliThoVas} obtained the sharp bounds of  $\vert T_{2,1}(f)\vert$, $\vert T_{3,1}(f)\vert$ and $\vert T_{3,2}(f)\vert$ for various subclasses of $\mathcal{S}$. In particular, they established the following result.
\begin{theoremA}\label{thmmA}\cite{AliThoVas}
       If $f\in \mathcal{S}^*$, then $\vert T_{3,2}(f)\vert \leq 84$. The inequality is sharp.
\end{theoremA}
      Ahuja et al.~\cite{AhuKhaRav} derived the bounds of  $\vert T_{2,1}(f)\vert$ and $\vert T_{3,1}(f)\vert$ for the classes $\mathcal{S}^*(\alpha)$ and $\mathcal{SS}^*(\beta)$.  Recently, Giri and Kumar~\cite{GirKum1} determined the sharp estimate of $\vert T_{2,1}(f)\vert$ formed over the logarithmic coefficients of functions belonging to $\mathcal{S}^*$ and $\mathcal{S}^*(\alpha)$. For further work on Toeplitz determinants for various classes, see~\cite{LecSimSmi,ObrTun}.

   As it is well known, one of the important and very interesting problems of classical and modern complex analysis is to find the extent in which the results of one dimensional theory can be generalized to several complex variables. 
   Motivated by this problem, Giri and Kumar~\cite{GirKum3} extended the study of Toeplitz determinants to higher dimensions.  They derived the sharp estimates of $\vert T_{2,1}(f)\vert$ and $\vert T_{3,1}(f)\vert$ for a class of holomorphic mappings defined on the unit ball in a complex Banach space and on the unit polydisk in $\mathbb{C}^n$. The work was further generalized for a class of holomorphic mappings with $k-$fold symmetric by Xu et al.~\cite{XuHeXu}. More recently, Giri~\cite{Gir1} established the sharp bounds of $\vert T_{2,2}(f)\vert$ for a class related to starlike mappings in several complex variables. In addition, the problems of estimating of $\vert T_{2,1}(f)\vert$ and $\vert T_{3,1}(f)\vert$  were studied for well known subclasses of holomorphic mappings in higher dimensions~\cite{GirKum4,XuJia}.

   In the existing literature, the problems of estimating $\vert T_{2,1}(f)\vert$ and $\vert T_{3,1}(f)\vert$ have been investigated on various domains, while the corresponding problem for $\vert T_{3,2}(f)\vert$ remained unaddressed. In this paper, we address this problem by establishing the sharp estimate of $\vert T_{3,2}(f)\vert$ for a subclass of holomorphic mappings defined on the unit ball in a complex Banach space. As special cases, our results yield bounds for subclasses of starlike mappings and lead to an extension of Theorem \hyperref[thmmA]{A} to several complex variables.

    In 1960, Zalcman proposed the conjecture $\vert a_n^2 - a_{2n-1} \vert \leq (n-1)^2 $  for $f \in \mathcal{S}$.  This conjecture has its own interest, but the main point is that it provides the Bieberbach conjecture $ |a_n|\leq n$. Both conjectures attracted considerable attention and were extensively studied by many authors. The Bieberbach conjecture has now been proved for all $n$, whereas Zalcman's conjecture is proved only for $n\leq 6$ and for certain subclasses of univalent functions~\cite{BroTsa,Kru,Ma2}.
    In 1999, Ma~\cite{Ma} proposed a generalized Zalcman conjecture for $f\in \mathcal{S}$ that
    $$ \vert a_n a_m - a_{n+m-1} \vert \leq (n-1)(m-1)  \quad \forall\; m,n\geq 2,$$
      which is still an open problem, however he proved it for the classes $\mathcal{S}^*$ and $\mathcal{S}_\mathbb{R}$, where $\mathcal{S}_\mathbb{R}$ contains univalent functions with real coefficients. Ravichandran and Verma~\cite{RavVer} established the conjecture for certain subclasses of $\mathcal{S}$. Furthermore, Cho et al.~\cite{ChoKwoLecSim} proved it for the classes $\mathcal{S}^*$, $\mathcal{S}^*(\alpha)$ and $\mathcal{SS}^*(\beta)$ when $n=2$ and $m=3$. The following theorem is a particular case of~\cite[Theorem 2.3]{Ma}.
\begin{theoremB}\cite{Ma}\label{thmmB}
    If $f(\zeta)\in  \zeta + \sum_{n=2}^\infty a_n \zeta^n \in \mathcal{S}^*$, then $\vert a_2 a_3 - a_4\vert \leq 2 $. The inequality is sharp.
\end{theoremB}
    Giri~\cite{Giri2} extended this estimate to a subclass of starlike mappings defined on the unit ball in a complex Banach space and on bounded starlike circular  domains in $\mathbb{C}^n$. In the present paper, we investigate the same problem for a class of holomorphic mappings on the unit ball in a complex Banach space. The derived result not only generalizes the work in~\cite{Giri2} but also yields sharp estimates for other subclasses of starlike mappings in several complex variables as special cases.

   Let $X$ be a complex Banach space with norm $\| \cdot\|$ and $\mathbb{B}$ represent the unit ball in $X$. Let $L(X,Y)$ denote the set of all continuous linear operators from $X$ into another complex Banach space $Y$. For each $z \in X\setminus\{0\}$, define the set
     $$ T(z) = \{ l_z \in L(X,\mathbb{C}) : l_z(z) = \| z\|, \| l_z \| = 1\}.$$
     The set $T(z)$ is non empty by the Hahn-Banach Theorem.
   Let $\mathcal{H}(\mathbb{B})$ represent the class of all holomorphic mappings from $\mathbb{B}$ into $X$. If $f\in \mathcal{H}(\mathbb{B})$, then for each $z\in \mathbb{B}$ and for all $w$ in some neighborhood of $z$, the function $f$ admits the expansion
    $$f(w)= \sum_{n=0}^\infty \frac{1}{n!} D^n f(z)((w-z)^n),$$
     where $D^n f(z)$ is  the $n^{th}$ Fr\'{e}chet derivative of $f$ at $z$, defined as a bounded symmetric $n-$linear mapping from $\prod_{j=1}^n X$ into $X$. Furthermore,
        $$ D^n f(z)((w-z)^n)= D^n f(z) \underbrace{( w-z, w-z, \cdots, w-z) }_\text{ n -times}.$$
   A mapping $f \in \mathcal{H}(\mathbb{B})$ is called biholomorphic if it has a holomorphic inverse defined on $f(\mathbb{B})$. It is said to be locally biholomorphic if the derivative $Df(z)$ is invertible with a bounded inverse for each $z \in \mathbb{B}$.
   Analogous to the one-dimensional case, the mapping $f$ is normalized if $f(0)=0$ and $Df(0)=I$, where $I$ is the identity operator on $X$. The collection of all normalized biholomorphic mappings on $\mathbb{B}$ is denoted by $\mathcal{S}(\mathbb{B})$.
    Hamada et al.~\cite{Ham4} defined the following class.
\begin{definition}~\cite{Ham4}
   Let $f: \mathbb{B} \rightarrow X$ be a normalized locally biholomorphic mapping and $\alpha \in (0,1)$. The mapping $f$ is said to be starlike of order $\alpha$ if
   $$ \left\vert \frac{1}{ \| z\|} l_z ([D f(z)]^{-1} f(z)) - \frac{1}{2 \alpha} \right\vert < \frac{1}{ 2 \alpha}, \quad \forall \;z \in \mathbb{B}\setminus\{0\}, \; l_z \in T(z) . $$
   When $\mathbb{B} = \mathbb{U}$ and $X = \mathbb{C}$,  the condition is equivalent to
   $$ \RE \Big( \frac{ \zeta f'(\zeta)}{f(\zeta)}\Big) > \alpha, \quad \zeta \in \mathbb{U}. $$
  We denote by $\mathcal{S}^*_\alpha (\mathbb{B})$ the class of all starlike mappings of order $\alpha$ on $\mathbb{B}$. In the case $\mathbb{B} = \mathbb{U}$ and $X = \mathbb{C}$, we write $\mathcal{S}^*_{\alpha}(\mathbb{U})=\mathcal{S}^*(\alpha).$
\end{definition}
\begin{definition}\cite{Kohr2}
   Let $f:\mathbb{B}\to X$ be a normalized locally biholomorphic mapping and $\beta\in(0,1]$. The mapping $f$ is said to be strongly starlike of order $\beta$ if
   $$ \left\vert \arg l_z ([D f(z)]^{-1} f(z))  \right\vert < \frac{\pi}{ 2} \beta, \quad \forall z \in \mathbb{B}\setminus\{0\}, \; l_z \in T(z) . $$
    When $\mathbb{B} = \mathbb{U}$ and $X = \mathbb{C}$,  the condition is equivalent to
   $$ \bigg\vert \arg \Big( \frac{ \zeta f'(\zeta)}{f(\zeta)}\Big)\bigg\vert < \frac{\pi}{ 2}  \beta, \quad \zeta \in \mathbb{U}. $$
  We denote by $\mathcal{SS}^*_\beta(\mathbb{B})$ the class of all strongly starlike mappings of order $\beta$ on $\mathbb{B}$. In the case $\mathbb{B}=\mathbb{U}$ and $X=\mathbb{C}$, we write $\mathcal{SS}^*_\beta(\mathbb{U})=\mathcal{SS}^*(\beta).$
\end{definition}
   For a biholomorphic function $\Phi : \mathbb{U} \rightarrow \mathbb{C}$ such that $\Phi(0)=1$ and $\RE \Phi(\zeta)>0$ on $\mathbb{U}$, Graham et al.~\cite{GraHamKoh} introduced the class
\begin{equation}\label{Mphi}
   \mathcal{M}_\Phi = \Big\{ p \in \mathcal{H}(\mathbb{B}) : p(0) =0, D(p(0))=I, \frac{\| z\|}{l_z ( p(z) )} \in \Phi(\mathbb{U}), z\in \mathbb{B}\setminus \{ 0 \} \Big\}.
\end{equation}
   In the case $\mathbb{B} = \mathbb{U}$ and $X = \mathbb{C}$,  this condition reduces to
    $$\mathcal{M}_\Phi = \left\{ p \in \mathcal{H}(\mathbb{U}) : p(0) =0, p'(0)=1, \frac{\zeta}{ p(\zeta) } \in \Phi(\mathbb{U}), \zeta \in \mathbb{U} \right\}.$$
  To establish the main results, we additionally assume throughout this paper that the function $\Phi$ satisfies $\Phi'(0)>0$ and  $\Phi(\mathbb{U})$ is symmetric with respect to the real axis, in addition to the conditions already imposed. Let the series expansion of $\Phi$ be given by
\begin{equation}\label{Phi}
   \Phi(\zeta)=1+B_1\zeta+B_2\zeta^2+B_3\zeta^3+\cdots,\qquad B_1>0.
\end{equation}
    Since $\Phi(\mathbb{U})$ is symmetric  and $\Phi(0)=1$, it follows that all the coefficients $B_i$ are real.
    It is worth noting that if $f\in \mathcal{H}(\mathbb{B})$ and $D (f(z))^{-1} f(z) \in \mathcal{M}_\Phi$, then specific choices of $\Phi$  in (\ref{Mphi}) lead to different subclasses of holomorphic mappings. For instance, taking
    $$\Phi(\zeta) = \frac{1+\zeta}{1 -\zeta}, \;\; \Phi(\zeta) = \frac{1 + (1-2 \alpha)\zeta}{1-\zeta} \;\; \text{and} \;\;\Phi(\zeta) = \Big(\frac{1+\zeta}{1-\zeta}\Big)^\beta ,$$
     we obtain that  $f \in \mathcal{S}^*(\mathbb{B})$, $f \in  \mathcal{S}^*_\alpha(\mathbb{B})$ and $f \in  \mathcal{SS}^*_\beta(\mathbb{B})$, respectively. Here, the branch of power function $((1+\zeta)/(1-\zeta))^\beta$ is chosen so that $((1+\zeta)/(1-\zeta))^\beta=1$ at $\zeta=0$.

\section{Third-order Toeplitz determinant}
   Let $\mathcal{B}_0$ represent the class of all Schwarz functions $\omega$ on $\mathbb{U}$ satisfying $\omega(0)=0$ and $\vert\omega(\zeta)\vert <1$ for all $\zeta\in \mathbb{U}$. The following lemmas are essential for proving the main results.
\begin{lemma}\label{lmm1}\cite{ProSzy}
     If $\omega(z)= \sum_{n=1}^\infty c_n z^n \in \mathcal{B}_0$, then
     $$\vert c_1\vert \leq 1 \;\; \text{and}\;\;  \vert c_2 \vert \leq 1 - \vert c_1\vert^2 .$$
\end{lemma}
\begin{lemma}\cite{ProSzy}\label{lemma2}
    If $\omega(z) = \sum_{n=1}^\infty c_n z^n \in \mathcal{B}_0$, then
\begin{equation*}
\begin{aligned}
\vert c_3 + q_1 c_1 c_2 +q_2 c_1^3 \vert \leq
\left\{
\begin{array}{ll}
     1   & \text{if}\;(q_1, q_2)\in \Theta_1\cup\Theta_2\cup\{(2,1)\},\\ \\
      \vert q_2 \vert & \text{if}\;(q_1, q_2)\in\cup_{i=3}^7 \Theta_i, \\ \\
      \dfrac{2}{3}(\vert q_1\vert +1)\Big(\dfrac{\vert q_1\vert+1}{3(\vert q_1\vert + q_2 +1)} \Big)^{1/2} & \text{if}\;(q_1, q_2)\in\Theta_8\cup\Theta_9,
\end{array}
\right.
\end{aligned}
\end{equation*}
    where $(q_1, q_2)\in \mathbb{R}^2$ and
\begin{align*}
  \Theta_1 &= \Big\{ (q_1 , q_2) : \vert q_1 \vert \leq \frac{1}{2}, \vert q_2\vert \leq 1 \Big\},  \\
   \Theta_2 &= \Big\{ (q_1 , q_2) : \frac{1}{2} \leq \vert q_1\vert \leq 2,\; \frac{4}{27} (\vert q_1\vert + 1)^3 - (\vert q_1\vert + 1) \leq q_2 \leq 1 \Big\}, \\
   \Theta_3 &= \Big\{ (q_1 , q_2) : \vert q_1\vert \leq \frac{1}{2}, q_2 \leq -1 \Big\}, \;\; \Theta_4 = \Big\{ (q_1 , q_2) : \vert q_1\vert \geq \frac{1}{2}, q_2 \leq -\frac{2}{3} (\vert q_1\vert + 1) \Big\}, \\
     \Theta_5 &= \Big\{ (q_1, q_2) : \vert q_1 \vert \leq 2, \; q_2 \geq 1 \Big\}, \;\;\Theta_6 = \Big\{ (q_1, q_2) : 2 \leq \vert q_1 \vert \leq 4, \; q_2 \geq \frac{1}{12} (q_1^2 + 8) \Big\},\\
     \Theta_7 &= \Big\{ (q_1, q_2) : \vert q_1 \vert \geq 4, \; q_2 \geq \frac{2}{3} (\vert q_1\vert - 1) \Big\}, \\
     \Theta_8 &=\Big\{ (q_1, q_2): \frac{1}{2} \leq \vert q_1\vert \leq 2,\; -\frac{2}{3} (\vert q_1 \vert +1 ) \leq q_2 \leq \frac{4}{27}(\vert q_1 \vert+ 1)^3 - (\vert q_1 \vert +1) \Big\},\\
      \Theta_9 &=\Big\{(q_1, q_2): \vert q_1 \vert \geq 2,\; -\frac{2}{3} (\vert q_1 \vert + 1) \leq q_2 \leq \frac{2 \vert q_1 \vert (\vert q_1 \vert + 1)}{q_1^2 + 2 \vert q_1 \vert + 4}\Big\}.
\end{align*}
   The estimate is sharp for the function $\omega(\zeta)= \zeta^3$ when $(q_1, q_2)\in \Theta_1\cup\Theta_2\cup\{(2,1)\}$ and for $\omega(\zeta)= \zeta$ when $(q_1, q_2)\in\cup_{i=3}^7 \Theta_i$.
\end{lemma}
    For $(q_1, q_2)\in \Theta_8 \cup \Theta_9$, Cho et al.~\cite{ChoKwoLecSim} obtained the explicit form of the extremal function, given by
   $$\omega(\zeta) = \frac{\zeta(\rho - \zeta)}{1 - \rho \zeta},$$
    where $$\rho = \sqrt{\frac{q_1+1}{3 (q_1 + q_2 +1)}}.$$
   In what follows, $\Phi$ refers to the function defined in (\ref{Phi}), unless otherwise stated.

\begin{lemma}\label{lm3}
   Let $h$ be a holomorphic function on $\mathbb{U}$ such that $h(\mathbb{U}) \subset \Phi(\mathbb{U})$ and $h(0)=\Phi(0)$. If $(q_1, q_2)\in \cup_{i=3}^7 \Theta_i$,  then 
\begin{align*}
    \Big\vert     \frac{ h'(0) h'''(0)}{3}  - \frac{3 (h'(0))^2 h''(0)}{2}    -  \frac{3 (h''(0))^2}{4}+  6 (h'(0))^2  -2 (h'(0))^4 \Big\vert \leq 2 B_1^2 \left\vert 3-  q_2 \right\vert,
\end{align*}
   where
\begin{equation}\label{q1q2}
    q_1 = -  \frac{(3 B_1^2 + 2 B_2)}{2 B_1} \;\; \text{and}\;\; q_2 =  -\frac{(2 B_1^4+3 B_1^2 B_2-2 B_1 B_3+3 B_2^2)}{2 B_1^2}.
\end{equation}
\end{lemma}
\begin{proof}
     By the hypothesis $h(0)=\Phi(0)$ and $h(\mathbb{U})\subset \Phi(\mathbb{U})$, we have $h\prec \Phi$. Therefore, there exists a Schwarz function $\omega(\zeta)=\sum_{n=1}^\infty c_n \zeta^n \in \mathcal{B}_0$ satisfying
    $$ h(\zeta) = \Phi(\omega(\zeta)) , \quad \zeta\in \mathbb{U}. $$
     A comparison of the coefficients of like powers of $\zeta$ in the expansions of $h$, $\Phi$ and $\omega$ yields
\begin{equation}\label{h1h2Lm3}
     h'(0)=  B_1 c_1,\;\;\frac{h''(0)}{2}= B_2 c_1^2+B_1 c_2 \;\; \text{and}   \;\; \frac{h'''(0)}{6}=  B_3 c_1^3 + 2 B_2 c_1 c_2 + B_1 c_3 .
\end{equation}
   A direct calculation using these expressions and the triangle inequality  gives
\begin{align*}
  \Big\vert     \frac{ h'(0) h'''(0)}{3}  - \frac{3 (h'(0))^2 h''(0)}{2}     - & \frac{3 (h''(0))^2}{4} +  6 (h'(0))^2 -2 (h'(0))^4 \Big\vert \\
                     &=   \Big\vert 6 B_1^2 c_1^2 - 3 B_1^2 c_2^2  + 2 B_1^2 c_1  \left( c_3 + c_1 c_2 q_1 + c_1^3 q_2 \right)\Big\vert \\
                     & \leq   6 B_1^2 \left\vert c_1 \right\vert^2 + 3 B_1^2 \left\vert c_2 \right\vert^2   + 2 B_1^2 \left\vert c_1 \right\vert  \left\vert c_3 + c_1 c_2 q_1 + c_1^3 q_2 \right\vert,
\end{align*}
   where $q_1$ and $q_2$ are given by (\ref{q1q2}).
    In view of Lemma~\ref{lmm1}, we get
\begin{align*}
     \Big\vert     \frac{ h'(0) h'''(0)}{3}  - \frac{3 (h'(0))^2 h''(0)}{2}     - & \frac{3 (h''(0))^2}{4} +  6 (h'(0))^2 -2 (h'(0))^4 \Big\vert \\
     & \leq  3 B_1^2 (1 + \left\vert c_1\right\vert^4)  + 2 B_1^2 \left\vert c_1 \right\vert  \left\vert c_3 + c_1 c_2 q_1 + c_1^3 q_2 \right\vert,
\end{align*}
   Since $(q_1, q_2) \in \cup_{i=3}^7 \Theta_i$, applying Lemma~\ref{lemma2} and  $\vert c_1\vert \leq 1$ from Lemma~\ref{lmm1}  to the above inequality, the required estimate follows.
\end{proof}

\begin{lemma}\label{lmA4}
 Let $h$ be a holomorphic function on $\mathbb{U}$ such that $h(\mathbb{U}) \subset \Phi(\mathbb{U})$ and $h(0)=\Phi(0)$. If $(q_3,q_4) \in \cup_{i=5}^7 \Theta_i $, then
    $$ \Bigg\vert \frac{(h'(0))^3}{2}+ \frac{3 h'(0) h''(0)}{4}  +\frac{h'''(0)}{6}  \Bigg\vert \leq \frac{1}{2} \left\vert B_1^3+3 B_2 B_1+2 B_3\right\vert,$$
    where
    $$q_3=\frac{2 B_2}{B_1}+\frac{3 B_1}{2} \;\;\text{and}\;\; q_4=\frac{B_1^3+3 B_1 B_2+2 B_3}{2 B_1}. $$
\end{lemma}
\begin{proof}
   From~(\ref{h1h2Lm3}), we obtain
\begin{align*}
     \Big\vert   \frac{h'''(0)}{6}+\frac{(h'(0))^3}{2}+ \frac{3 h'(0) h''(0)}{4} \Big\vert =  B_1\left\vert c_3 + q_3 c_1 c_2  +  q_4 c_1^3   \right\vert.
\end{align*}
   Using the bound from Lemma~\ref{lemma2}, we deduce the stated result.
\end{proof}
    The following result provides the sharp bound of $\vert T_{3,2}(F)\vert$ in terms of the corresponding homogeneous expansion of a biholomorphic mapping $F$. For $F\in \mathcal{H}(\mathbb{B})$, we use the following notation throughout this section:
\begin{equation}\label{A2A3A4F}
   A_2=\frac{ l_z (D^2 F(0) (z^2))}{2! \vert\vert z \vert\vert^2},  \;\; A_3= \frac{ l_z (D^3 F(0) (z^3))}{3! \vert\vert z \vert\vert^3}\;\; \text{and}  \;\; A_4= \frac{ l_z (D^4 F(0) (z^4))}{4! \vert\vert z \vert\vert^4}.
\end{equation}
\begin{theorem}\label{thmB}
    Let $f\in \mathcal{H}(\mathbb{B},\mathbb{C})$ with $f(0)=1$ and define $F(z)= z f(z)$. If $(D F(z))^{-1} F(z) \in \mathcal{M}_\Phi$ such that $(q_1, q_2)\in \cup_{i=3}^7 \Theta_i$ and $(q_3,q_4) \in \cup_{i=5}^7 \Theta_i $,
   then
\begin{equation*}
\begin{aligned}
    \vert A_2^3-2 A_2 A_3^2- A_2 A_4^2 + 2 A_3^2 A_4 \vert \leq  \frac{( B_1^3  + 3 B_1 B_2 + 2 B_3 + 6 B_1 ) \vert  2 B_1^4+3 B_1^2 B_2+6 B_1^2 -2 B_1 B_3 +3 B_2^2  \vert }{36} ,
\end{aligned}
\end{equation*}
   where
\begin{align*}
       q_1 &= -\frac{(3 B_1^2 + 2 B_2)}{2 B_1}, \quad  q_2 =  -\frac{(2 B_1^4+3 B_1^2 B_2-2 B_1 B_3+3 B_2^2)}{2 B_1^2},\\
        q_3 &= \frac{3 B_1^2 + 4 B_2}{2 B_1}, \quad \quad \;  q_4=\frac{B_1^3+3 B_1 B_2+2 B_3}{2 B_1}
\end{align*}
    and $A_2$, $A_3$, $A_4$ are given by (\ref{A2A3A4F}). Moreover, the estimate is sharp.
\end{theorem}
\begin{proof}
     Let $z_0 = \frac{z}{\|z \|}$ for a fixed $z\in X\setminus \{ 0 \}$.  Consider the function $ h : \mathbb{U} \rightarrow \mathbb{C}$ given by
\begin{equation*}
    h(\zeta) = \left\{ \begin{array}{ll}
     \dfrac{\zeta}{ l_z ((D F(\zeta z_0))^{-1} F( \zeta z_0) )}, & \zeta \neq 0, \\ \\
    1, & \zeta =0.
    \end{array}
    \right.
\end{equation*}
   Clearly, $h \in \mathcal{H}(\mathbb{U})$ and $h(0)=1  = \Phi(0)$. Furthermore, since $(D F(z))^{-1} F(z) \in \mathcal{M}_\Phi$, it follows
   that
\begin{align*}
   h(\zeta) = \frac{\zeta}{l_z ((D F(\zeta z_0))^{-1} F( \zeta z_0) ) } = &\frac{\zeta}{l_{z_0} ((D F(\zeta z_0))^{-1} F( \zeta z_0) ) } \\
             =& \frac{\| \zeta z_0 \| }{l_{ \zeta z_0} ((D F(\zeta z_0))^{-1} F( \zeta z_0) ) } \in \Phi(\mathbb{U}), \quad  \zeta \in \mathbb{U}.
\end{align*}
   Consequently, this implies that $h\prec \Phi$. On the other hand, following the approach adopted in~\cite[Theorem 7.1.14]{GraKoh} (also see \cite[Theorem 3.2]{XuLiuLiu}), we obtain
    $$ (D F(z))^{-1} = \frac{1}{f(z)} \bigg( I - \frac{\frac{z D f(z)}{f(z)}}{1 + \frac{D f(z) z}{f(z)}} \bigg),  \quad z\in \mathbb{B}. $$
    In view of the above, we get
    $$ (D F(z))^{-1} F(z) = \frac{z f(z) }{f(z) + D f(z) z} ,  $$
    which immediately yields
\begin{equation}\label{newe}
   \frac{\| z\|}{l_z ((D F(z))^{-1} F(z))} = 1 + \frac{D f(z) z}{f(z)} .
\end{equation}
   By virtue of (\ref{newe}), we deduce that
\begin{equation}\label{accr}
    h(\zeta) = \frac{\| \zeta z_0 \| }{l_{ \zeta z_0} ((D F(\zeta z_0))^{-1} F( \zeta z_0) ) }  =  1 + \frac{D f(\zeta z_0)\zeta z_0}{f(\zeta z_0)}.
\end{equation}
   The Taylor series expansions of $h(\zeta)$ and $f(\zeta z_0)$ yield
\begin{align*}
   \bigg(1 + & h'(0) \zeta  + \frac{h''(0)}{2} \zeta^2 + \cdots \bigg)\bigg( 1 + Df(0)(z_0) \zeta + \frac{ D^2 f(0)(z_{0}^2)}{2} \zeta^2 + \cdots \bigg)  \\
   & =\bigg( 1 + Df(0)(z_0) \zeta + \frac{ D^2 f(0)(z_{0}^2)}{2} \zeta^2 + \cdots \bigg)+ \bigg( Df(0)(z_0) \zeta +  D^2 f(0)(z_{0}^2)\zeta^2 + \cdots \bigg).
\end{align*}
  Comparing the homogeneous terms, we obtain
\begin{equation*}\label{eqhf}
\begin{aligned}
    h'(0) &= D f(0)(z_0),\;\; \dfrac{h''(0)}{2} =  D^2 f(0)(z_0^2) - (D f(0)(z_0))^2,\\
    \dfrac{h'''(0)}{6} &= ( Df(0)(z_0))^3 - \dfrac{3}{2} D f(0)(z_0) D^2 f(0) (z_0^2) + \dfrac{D^3 f(0) (z_0^3)}{2}.
\end{aligned}
\end{equation*}
   That is
\begin{equation}\label{eqhf2}
\begin{aligned}
\left.
\begin{array}{ll}
     h'(0) \|z\|&= D f(0)(z), \;\;    \dfrac{h''(0)}{2}\| z\|^2 =  D^2 f(0)(z^2) - (D f(0)(z))^2, \\ \\
    \dfrac{h'''(0)}{6} \|z\|^3 &=  ( Df(0)(z))^3 -\dfrac{3}{2} D f(0)(z) D^2 f(0) (z^2) + \dfrac{D^3 f(0) (z^3)}{2}.
\end{array}
\right\}
\end{aligned}
\end{equation}
     Since $F(z) = z f(z)$, it follows that
\begin{align*}
     \frac{ D^2 F(0) (z^2)}{2! } =  D f(0)(z)  z,  \; \;    \frac{ D^3 F(0) (z^3)}{3! } =  \frac{ D^2 f(0) (z^2)}{2! } z \;\;\text{and}\;\;
      \frac{ D^4 F(0) (z^4)}{4! } =  \frac{ D^3 f(0) (z^3)}{3! } z,
\end{align*}
   which leads to
    $$ \frac{l_z( D^2 F(0) (z^2))}{2! } =  D f(0)(z)  \|z\|, \quad   \frac{l_z (D^3 F(0) (z^3))}{3! } =  \frac{ D^2 f(0) (z^2)}{2! } \|z\| $$
   and
   $$          \frac{l_z( D^4 F(0) (z^4))}{4! } =  \frac{ D^3 f(0) (z^3)}{3! } \|z\|,$$
   respectively. Using (\ref{eqhf2}), we deduce that
\begin{equation}\label{A4Expre}
\left.
\begin{aligned}
     \frac{l_z( D^2 F(0) (z^2))}{2! \|z\|^2} &= h'(0), \\
     \frac{l_z (D^3 F(0) (z^3))}{3! \|z\|^3} &= \frac{1}{2}\bigg(\frac{h''(0)}{2}+ (h'(0))^2 \bigg),\\
      \frac{l_z( D^4 F(0) (z^4))}{4! \|z\|^4} &=  \frac{1}{3} \bigg( \frac{h'''(0)}{6}+\frac{(h'(0))^3}{2}+ \frac{3 h'(0) h''(0)}{4} \bigg).
\end{aligned}
\right\}
\end{equation}
  From~(\ref{A4Expre}) together with~(\ref{A2A3A4F}), we get
\begin{align*}
     \Big\vert A_2^2 -2 A_3^2 + A_2 A_4 \Big\vert &= \frac{1}{6} \Big\vert  6 (h'(0))^2- \frac{3 (h'(0))^2 h''(0)}{2}  +   \frac{ h'(0) h'''(0)}{3} -  \frac{3 (h''(0))^2}{4}  -2 (h'(0))^4 \Big\vert.
\end{align*}
    By the hypothesis $(q_1, q_2)\in \cup_{i=3}^7 \Theta_i$, applying Lemma~\ref{lm3} to the above equality yields
\begin{equation}\label{A2A3A4}
     \Big\vert A_2^2 -2 A_3^2 + A_2 A_4 \Big\vert  \leq \frac{1}{6} \left\vert 2 B_1^4+3 B_1^2 B_2+6 B_1^2-2 B_1 B_3+3 B_2^2\right\vert.
\end{equation}
  Further, by~(\ref{A4Expre}), we have
  $$ \bigg\vert \frac{l_z( D^4 F(0) (z^4))}{4! \|z\|^4} \bigg\vert = \bigg\vert  \frac{1}{3} \bigg( \frac{h'''(0)}{6}+\frac{(h'(0))^3}{2}+ \frac{3 h'(0) h''(0)}{4} \bigg) \bigg\vert. $$
 Combining the hypothesis $(q_3, q_4)\in \cup_{i=5}^7 \Theta_i$ with Lemma~\ref{lmA4}, we obtain
\begin{equation}\label{BdA4}
   \vert A_4\vert  =  \bigg\vert \frac{l_z( D^4 F(0) (z^4))}{4! \|z\|^4} \bigg\vert \leq   \frac{1}{6}  \left\vert B_1^3+3 B_1  B_2+2 B_3\right\vert .
\end{equation}
  Since $h\prec \Phi$, we have $\vert h'(0)\vert \leq B_1$. Therefore, by~(\ref{A4Expre}), it follows that
\begin{equation}\label{BdA2}
  \vert A_2\vert  =   \bigg\vert \frac{l_z( D^2 F(0) (z^2))}{2! \|z\|^2} \bigg\vert \leq B_1.
\end{equation}
  Employing the bounds in~(\ref{A2A3A4}), (\ref{BdA4}) and~(\ref{BdA2}) together with the triangle inequality, we derive
\begin{align*}
   \left\vert T_{3,2}(F) \right\vert &= \left\vert  A_2^3-2 A_2 A_3^2- A_2 A_4^2 + 2 A_3^2 A_4 \right\vert \\
             &\leq  \left (\vert A_2\vert + \vert A_4\vert \right) \left\vert A_2^2 -2 A_3^2 + A_2 A_4 \right\vert\\
        &\leq \frac{( B_1^3 + 3 B_1 B_2 + 2 B_3 + 6 B_1 ) \left\vert ( 2 B_1^4+3 B_1^2 B_2+6 B_1^2-2 B_1 B_3+3 B_2^2 ) \right\vert}{36}.
\end{align*}

    In order to establish the sharpness of the estimate, let us consider the mapping $F$ defined by
\begin{equation}\label{extB}
    F(z) = z \exp \int_0^{l_u(z)} \frac{( \Phi(i t)-1) }{t}dt, \quad z\in \mathbb{B}, \quad \vert\vert u \vert\vert=1,
\end{equation}
 where $\Phi(\zeta)$ is the same as that given in~(\ref{Phi}).
We deduce that $(D F(z))^{-1}F(z)\in \mathcal{M}_\Phi$ and a straightforward computation shows that
  $$  \frac{D^2  F(0) (z^2)}{2!}= i B_1 l_u(z) z ,\quad  \frac{D^3  F(0) (z^3)}{3!} = -\frac{1}{2} \left( B_1^2 +  B_2  \right) (l_u (z))^2 z$$
     \text{and}
\begin{align*}
     \frac{D^4  F(0) (z^4)}{4!} & =-\frac{i}{6} \left( B_1^3 + 3 B_1 B_2  + 2 B_3 \right)(l_u(z))^3 z,
\end{align*}
   which immediately provide
  $$  \frac{l_z(D^2  F(0) (z^2))}{2!}= i B_1 l_u(z) \|z\|, \;\;\; \frac{l_z (D^3  F(0) ( z^3 ))}{3!} =  -\frac{1}{2} \left( B_1^2 +  B_2  \right) (l_u (z))^2 \| z \|,$$
   and
\begin{align*}
     \frac{l_z (D^4  F(0) (z^4))}{4!} & =-\frac{i}{6} \bigg( B_1^3 + 3 B_1 B_2  + 2 B_3 \bigg)(l_u(z))^3 \|z\|,
\end{align*}
   respectively. Setting $z = r u$ $(0< r <1)$, we obtain
\begin{equation}\label{cftBToep}
\left.
\begin{aligned}
     \frac{l_z(D^2  F(0) (z^2))}{2! \|z\|^2}&=  i B_1,\\
     \frac{l_z (D^3  F(0) ( z^3) ) }{3! \| z \|^3}  &=   -\frac{1}{2} \left( B_1^2 +  B_2  \right), \\
      \frac{l_z (D^4  F(0) (z^4))}{4! \| z \|^4} & =-\frac{i}{6} \left( B_1^3 + 3 B_1 B_2  + 2 B_3 \right).
\end{aligned}
\right\}
\end{equation}
   Consequently, in view of (\ref{A2A3A4F}) and (\ref{cftBToep}), for the mapping $F$ defined in (\ref{extB}), it follows that
\begin{equation*}
\begin{aligned}
  \vert A_2^3-2A_2A_3^2-A_2A_4^2+2A_3^2A_4\vert =\frac{ (B_1^3+3B_1B_2+2B_3+6B_1) \vert 2B_1^4+3B_1^2B_2+6B_1^2-2B_1B_3+3B_2^2 \vert}{36},
\end{aligned}
\end{equation*}
  which confirms the sharpness of the bound.
\end{proof}
\section{Special cases}
   If we replace $\Phi(\zeta)$  by $(1+\zeta)/(1 -\zeta)$, $(1 + (1-2 \alpha)\zeta)/(1-\zeta)$ and $((1+\zeta)/(1-\zeta))^\beta$
    in Theorem~\ref{thmB}, we obtain the following bounds for the corresponding subclasses of $\mathcal{S}(\mathbb{B})$, respectively.
\begin{corollary}
    Let $f: \mathbb{B} \rightarrow \mathbb{C}$ with $f(0)=1$ and $F(z)= z f(z) \in \mathcal{S}^*(\mathbb{B})$. Then
\begin{equation*}
\begin{aligned}
    \Big\vert A_2^3-2 A_2 A_3^2- A_2 A_4^2 + 2 A_3^2 A_4 \Big\vert &\leq  84,
\end{aligned}
\end{equation*}
   where
    $A_2$, $A_3$, $A_4$ are given by (\ref{A2A3A4F}). The  estimate is sharp.
\end{corollary}

\begin{corollary}
    Let $f: \mathbb{B} \rightarrow \mathbb{C}$ with $f(0)=1$ and $F(z)= z f(z) \in \mathcal{S}^*_\alpha(\mathbb{B})$. Then
  $$   \Big\vert A_2^3-2 A_2 A_3^2- A_2 A_4^2 + 2 A_3^2 A_4 \Big\vert \leq  \frac{4}{9} (1-\alpha )^3 (2 \alpha ^2-7 \alpha +9 ) (8 \alpha ^2-22 \alpha +21 )  $$
  for  $\alpha \in [0,\alpha_0]$,  where $\alpha_0 = (27-\sqrt{129})/24\approx 0.651758$ is a root of
\begin{equation*}
      24 \alpha ^2 - 54 \alpha +25=0
\end{equation*}
   and  $A_2$, $A_3$, $A_4$ are given by (\ref{A2A3A4F}).   The estimate is sharp.
\end{corollary}
\begin{corollary}
    Let $f: \mathbb{B} \rightarrow \mathbb{C}$ with $f(0)=1$ and $F(z)= z f(z) \in \mathcal{SS}^*_\beta(\mathbb{B})$. Then
 $$   \Big\vert A_2^3-2 A_2 A_3^2- A_2 A_4^2 + 2 A_3^2 A_4 \Big\vert \leq  \frac{4}{81} \beta ^3 (17 \beta ^2+10 ) (47 \beta ^2+16 )$$
   for $\beta \in [\beta_0,1]$,  where $\beta_0 = (8 + \sqrt{346})/47\approx 0.56598$ is a root of
\begin{equation*}
       47 \beta ^2 -16 \beta -6 =0
\end{equation*}
   and  $A_2$, $A_3$, $A_4$ are given by (\ref{A2A3A4F}).  The estimate is sharp.
\end{corollary}
\section{Generalized Zalcman conjecture}
   In this section, we extend the bound given in Theorem \hyperref[thmmB]{B}  to several complex variables.  To establish the desired results, we use the following lemma. Unless otherwise stated, $\Phi$ denotes the function given by \eqref{Phi} throughout the section.
\begin{lemma}\label{Lmm7}
Let $h$ be a holomorphic function on $\mathbb{U}$ such that $h(\mathbb{U}) \subset \Phi(\mathbb{U})$ and $h(0)=\Phi(0)$. Then
\begin{align*}
    \Big\vert (h'(0))^3 -\frac{h'''(0)}{6} \Big\vert \leq
\left\{
\begin{array}{ll}
      B_1 & \text{if} \; (q_5,q_6) \in  \Theta_1\cup\Theta_2\cup\{(2,1)\},   \\ \\
      \vert B_3-  B_1^3   \vert & \text{if} \; (q_5,q_6) \in \cup_{i=3}^{7} \Theta_i ,  \\ \\
        \dfrac{2 \sqrt{3}(B_1 + 2 \vert B_2\vert)^{3/2}}{9\sqrt{B_1 - B_1^3 + B_3 + 2 \vert B_2\vert}} & \text{if}\; (q_5,q_6) \in  \Theta_8\cup\Theta_9,
\end{array}
\right.
\end{align*}
   where
\begin{equation*}\label{q5q6}
    q_5 =\frac{2 B_2}{B_1} \;\; \text{and}\;\; q_6 =  \frac{B_3- B_1^3}{B_1}.
\end{equation*}
\end{lemma}
\begin{proof}
    In view of (\ref{h1h2Lm3}), we obtain
\begin{align*}
  \left\vert (h'(0))^3  - \frac{h'''(0)}{6}  \right\vert &=  B_1 \left\vert c_3 + q_5 c_1 c_2 + q_6 c_1^3  \right\vert.
\end{align*}
   The required estimate follows by applying Lemma~\ref{lemma2} to the above equation.
\end{proof}

\begin{theorem}\label{thmZalc}
   Let $f\in \mathcal{H}(\mathbb{B},\mathbb{C})$ with $f(0)=1$ and define $F(z)= z f(z)$. If $(D F(z))^{-1} F(z) \in \mathcal{M}_\Phi$,   then
\begin{equation*}
      \left\vert A_2 A_3 - A_4 \right\vert  \leq
\begin{aligned}
\left\{
\begin{array}{ll}
      \dfrac{B_1}{3} & \text{if} \; (q_5,q_6) \in  \Theta_1\cup\Theta_2\cup\{(2,1)\},   \\ \\
      \dfrac{1}{3} \left\vert B_3-  B_1^3   \right\vert & \text{if} \; (q_5,q_6) \in \cup_{i=3}^{7} \Theta_i ,  \\ \\
       \dfrac{2 \sqrt{3}(B_1 + 2 \vert B_2\vert)^{3/2}}{27\sqrt{B_1 - B_1^3 + B_3 + 2 \vert B_2\vert}}  & \text{if}\; (q_5,q_6) \in  \Theta_8\cup\Theta_9,
\end{array}
\right.
\end{aligned}
\end{equation*}
  where $q_5 ={2 B_2}/{B_1},$  $q_6 =  {(B_3- B_1^3)}/{B_1} $
   and $A_2$, $A_3$, $A_4$ are given by (\ref{A2A3A4F}). Moreover, the above estimate is sharp.
\end{theorem}
\begin{proof}
   Following the same methodology employed in the proof of Theorem~\ref{thmB}, and combining~(\ref{A4Expre}) with~(\ref{A2A3A4F}), we get
\begin{align*}
     \left\vert A_2 A_3 - A_4 \right\vert   = \frac{1}{3} \bigg\vert (h'(0))^3 - \frac{h'''(0)}{6} \bigg\vert.
\end{align*}
   Applying Lemma~\ref{Lmm7} to the preceding equation, the required estimate follows immediately.

    In order to prove the sharpness, let us define the mapping $\tilde{F}_1$ as
\begin{equation}\label{extBZalc}
    \tilde{F}_1(z) = z \exp \int_0^{l_u(z)} \frac{( \Phi(t^3)-1) }{t}dt, \quad z\in \mathbb{B}, \quad \vert\vert u \vert\vert=1.
\end{equation}
  Note that $ (D  \tilde{F}_1(z))^{-1}  \tilde{F}_1(z) \in \mathcal{M}_\phi$ and a straightforward computation reveals
  $$  \frac{D^2  \tilde{F}_1(0) (z^2)}{2!}=  0, \; \;  \frac{D^3  \tilde{F}_1(0) (z^3)}{3!} = 0\;\; \text{and}\;\;  \frac{D^4  \tilde{F}_1(0) (z^4)}{4!}  =\frac{1}{3} B_1 (l_u(z))^3 z, $$
   which directly yield
  $$  \frac{l_z(D^2  \tilde{F}_1(0) (z^2))}{2!}= 0, \;\; \frac{l_z (D^3  \tilde{F}_1(0) ( z^3 ))}{3!} = 0 \;\;\text{and}\;\; \frac{l_z (D^4  \tilde{F}_1(0) (z^4))}{4!}  =\frac{1}{3} B_1 (l_u(z))^3 \|z\|,$$
   respectively. Upon setting  $z = r u$ $(0< r <1)$, we obtain
\begin{equation}\label{cftBZalc}
     \frac{l_z(D^2  \tilde{F}_1(0) (z^2))}{2! \|z\|^2}=  0, \; \frac{l_z (D^3 \tilde{F}_1(0) ( z^3) ) }{3! \| z \|^3}  =0 \; \text{and}\; \frac{l_z (D^4  \tilde{F}_1(0) (z^4))}{4! \| z \|^4}  =\frac{1}{3} B_1.
\end{equation}
   Thus, combining (\ref{A2A3A4F}) and (\ref{cftBZalc}), for the mapping $\tilde{F}_1$ defined in (\ref{extBZalc}), we deduce that
\begin{align*}
      \left\vert A_2 A_3 - A_4 \right\vert  =  \frac{B_1}{3},
\end{align*}
   which confirms the sharpness of the bound when $(q_5,q_6)\in \Theta_1\cup\Theta_2\cup\{(2,1)\}$.
   Using the same argument, one can verify that equality is attained by the mappings $\tilde{F}_2$ and $\tilde{F}_3$ for $(q_5,q_6)\in \bigcup_{i=3}^{7}\Theta_i$ and $(q_5,q_6)\in \Theta_8\cup\Theta_9$, respectively, which are defined by
\begin{equation}\label{F2Ext}
    \tilde{F}_2(z) = z \exp \int_0^{l_u(z)} \frac{( \Phi(t)-1) }{t}dt, \quad z\in \mathbb{B}, \quad \vert\vert u \vert\vert=1
\end{equation}
   and
\begin{equation}\label{F3Ext}
    \tilde{F}_3(z) = z \exp \int_0^{l_u(z)} \frac{( \Phi(\omega(t))-1) }{t}dt, \quad z\in \mathbb{B}, \quad \vert\vert u \vert\vert=1,
\end{equation}
   where
    $$\omega(\zeta)=\frac{\zeta(\rho-\zeta)}{1-\rho \zeta}\; \text{and}\; \rho=\sqrt{\frac{B_1+2 B_2}{3 ( B_1 + 2 B_2 + B_3 -B_1^3)}}, \quad \zeta \in \mathbb{U}.$$
     Consequently, the estimate is sharp in all cases.
\end{proof}
\subsection{Special cases}
  By taking
  $\phi(\zeta)=(1+\zeta)/(1-\zeta)$, $\phi(\zeta)=(1+(1-2\alpha)\zeta)/(1-\zeta)$ and $\phi(\zeta)=((1+\zeta)/(1-\zeta))^\beta$ in Theorem~\ref{thmZalc}, we obtain the following bounds for the classes $\mathcal{S}^*(\mathbb{B})$, $\mathcal{S}^*_\alpha(\mathbb{B})$ and $\mathcal{SS}^*_\beta(\mathbb{B})$, respectively.
\begin{corollary}
    Let $f: \mathbb{B} \rightarrow \mathbb{C}$ with $f(0)=1$ and $F(z)= z f(z) \in \mathcal{S}^*(\mathbb{B})$. Then
\begin{equation*}
\begin{aligned}
    \left\vert A_2 A_3 - A_4 \right\vert &\leq 2,
\end{aligned}
\end{equation*}
   where
    $A_2$, $A_3$, $A_4$ are given by (\ref{A2A3A4F}). Equality is attained by the mapping
    $$ \tilde{F}_2(z) =\frac{z}{(1-l_u(z))^2}, \quad z\in \mathbb{B}, \quad \vert\vert u \vert\vert=1. $$
    Hence, the bound is sharp.
\end{corollary}
\begin{corollary}
    Let $f: \mathbb{B} \rightarrow \mathbb{C}$ with $f(0)=1$ and $F(z)= z f(z) \in \mathcal{S}^*_\alpha(\mathbb{B})$. Then
\begin{align*}
      \left\vert A_2 A_3 - A_4 \right\vert \leq
\left\{
\begin{array}{ll}
        \dfrac{2(3 - 11 \alpha + 12 \alpha^2 - 4 \alpha^3)}{3}      &\;\; \text{if} \; \alpha \in (0,\alpha_1], \\ \\
        \dfrac{2 (1-\alpha)}{3 \sqrt{\alpha (2-\alpha)}}   &\;\;\text{if} \; \alpha \in [\alpha_1,1),
\end{array}
\right.
\end{align*}
    where $\alpha_1 = (2-\sqrt{3})/4\approx 0.133975$ is a root of the equation
\begin{equation*}
      4 \alpha^2 -8 \alpha  + 1 =0
\end{equation*}
   and  $A_2$, $A_3$, $A_4$ are given by (\ref{A2A3A4F}).   When $\alpha \in (0,\alpha_1]$, it follows that $(q_5,q_6)\in \Theta_4$. Therefore, equality case holds for the mapping
   $$\tilde{F}_2(z)=  \frac{z}{\left(1-l_u(z)\right)^{2(1-\alpha)}} , \;\; \;\; z\in \mathbb{B},\;\; \vert\vert u \vert\vert=1 .$$
    For  $\alpha \in (\alpha_0,1)$, we have $(q_5,q_6)\in \Theta_8$. In this case, equality is attained by the mapping
   $$ \tilde{F}_3(z) =   z \exp \int_0^{l_u(z)} \bigg(\frac{2 (1- \alpha ) (\rho - t )}{t^2-2 t  \rho +1}\bigg) dt, \;\; \;\; z\in \mathbb{B},\;\; \vert\vert u \vert\vert=1 ,$$
  where  $ \rho = {1}/{(2 \sqrt{ \alpha (2   -\alpha)}})$.
   Hence, the bound is sharp.
\end{corollary}
\begin{corollary}
    Let $f: \mathbb{B} \rightarrow \mathbb{C}$ with $f(0)=1$ and $F(z)= z f(z) \in \mathcal{SS}^*_\beta(\mathbb{B})$. Then
\begin{align*}
      \left\vert A_2 A_3 - A_4 \right\vert \leq
\left\{
\begin{array}{ll}
        \dfrac{2\beta}{3}     & \text{if} \;\beta \in (0,\beta_1] \\ \\
         \dfrac{2 \sqrt{2} \beta  (2 \beta +1)^{3/2}}{9 \sqrt{ 2 + 3 \beta - 5 \beta^2}}    & \text{if} \;\beta \in [\beta_1,\dfrac{2+\sqrt{34}}{10}], \\ \\
         \dfrac{2\beta( 10 \beta ^2 - 1)}{9} & \text{if} \;\beta \in  [\dfrac{2+\sqrt{34}}{10},1]
\end{array}
\right.
\end{align*}    where $\beta_1 \approx 0.559376$ is a root of the equation
\begin{equation*}
       16 \beta ^3+69 \beta ^2-15 \beta -16=0
\end{equation*}
   and  $A_2$, $A_3$, $A_4$ are given by (\ref{A2A3A4F}).
 For $\beta \in (0,\beta_1]$, we have  $(q_5,q_6)\in \Theta_1\cup \Theta_2$. The equality is attained by the mapping
   $$\tilde{F}_1(z)=  z \exp \int_0^{l_u(z)} \frac{1}{t}\Big(\Big(\frac{1+t^3}{1-t^3} \Big)^\beta -1 \Big) dt . $$
 For $\beta\in [\beta_1,(2+\sqrt{34})/10]$, we have $(q_5,q_6)\in \Theta_8.$ In this case, the extremal mapping is
 $$\tilde{F}_3 = z \exp \int_0^{l_u(z)} \frac{1}{t}\Big(\Big(\frac{1- t^2}{1+ t^2 - 2\rho t}\Big)^\beta -1 \Big) dt   ,$$
  where  $ \rho= \sqrt{{(1 + 2 \beta) }/{( 4  +6 \beta -10 \beta^2)}}.$
 Moreover, for $\beta \in [(2+\sqrt{34})/10,1]$, we have $(q_5,q_6)\in \Theta_3$. Consequently, the extremal mapping is given by
 $$\tilde{F}_2(z)= z \exp \int_0^{l_u(z)} \frac{1}{t}\Big(\Big(\frac{1+t}{1-t} \Big)^\beta -1 \Big) dt . $$
  Hence, the estimate is sharp for all $\beta \in (0,1]$.
\end{corollary}
\section*{Declarations}

\subsection*{Conflict of interest}
	The author declare that he has no conflict of interest.
\subsection*{Author Contribution}
    Each author contributed equally to the research and preparation of the manuscript.
\subsection*{Data Availability} Not Applicable.


\begin{thebibliography}{99}
\bibitem{AhuKhaRav} O.~P. Ahuja, K. Khatter and V. Ravichandran, Toeplitz determinants associated with Ma-Minda classes of starlike and convex functions, Iran. J. Sci. Technol. Trans. A Sci. {\bf 45} (2021), no.~6, 2021--2027.

\bibitem{AliThoVas} M.~F. Ali, D.~K. Thomas and A. Vasudevarao, Toeplitz determinants whose elements are the coefficients of analytic and univalent functions, Bull. Aust. Math. Soc. {\bf 97} (2018), no.~2, 253--264.

\bibitem{BroTsa} J.~E. Brown and A. Tsao, On the Zalcman conjecture for starlike and typically real functions, Math. Z. {\bf 191} (1986), no.~3, 467--474.

\bibitem{ChoKwoLecSim}  N. E. Cho,  O. S. Kwon,  A. Lecko and Y. J. Sim, Sharp estimates of generalized Zalcman functional of early coefficients for Ma-Minda type functions, Filomat {\bf 32} (2018), no.~18, 6267--6280.

\bibitem{Gir1} S. Giri, Second-order Toeplitz determinant for starlike mappings in one and higher dimensions, Anal. Math. Phys. {\bf 15} (2025), no.~4, Paper No. 96, 17 pp.

\bibitem{Giri2} S. Giri, Generalized Zalcman Conjecture for Starlike Mappings in Several Complex Variables, arXiv preprint arXiv:2601.22558 (2026).

\bibitem{GirKum1} S. Giri and S.~S. Kumar, Toeplitz determinants of logarithmic coefficients for starlike and convex functions, Bull. Sci. Math. {\bf 211} (2026), Paper No. 103845, 17 pp.


\bibitem{GirKum3} S. Giri and S.~S. Kumar, Toeplitz determinants in one and higher dimensions, Acta Math. Sci. Ser. B (Engl. Ed.) {\bf 44} (2024), no.~5, 1931--1944.

\bibitem{GirKum4} S. Giri and S.~S. Kumar, Toeplitz determinants for a class of holomorphic mappings in higher dimensions, Complex Anal. Oper. Theory {\bf 17} (2023), no.~6, Paper No. 86, 16 pp.

\bibitem{Goodman} A.W. Goodman, Univalent Functions, Mariner, Tampa (1983).

\bibitem{GraHamKoh} I.~R. Graham, H. Hamada and G. Kohr, Parametric representation of univalent mappings in several complex variables, Canad. J. Math. {\bf 54} (2002), no.~2, 324--351.

\bibitem{GraKoh} I. Graham and G. Kohr. Geometric Function Theory in One and Higher Dimensions. Vol. 255. Monographs and Textbooks in Pure and Applied Mathematics. New York: Marcel Dekker, Inc., 2003.

\bibitem{Ham4} H. Hamada, G. Kohr\ and\ P. Liczberski, Starlike mappings of order $\alpha$ on the unit ball in complex Banach spaces, Glas. Mat. Ser. III {\bf 36(56)} (2001), no.~1, 39--48.

\bibitem{Kohr2} G. Kohr\ and\ P. Liczberski, On strongly starlikeness of order alpha in several complex variables, Glas. Mat. Ser. III {\bf 33(53)} (1998), no.~2, 185--198.

\bibitem{Kru} S.~L. Krushkal, Proof of the Zalcman conjecture for initial coefficients, Georgian Math. J. {\bf 17} (2010), no.~4, 663--681.

\bibitem{LecSimSmi} A. Lecko, Y.~J. Sim and B. \'Smiarowska, The fourth-order Hermitian Toeplitz determinant for convex functions, Anal. Math. Phys. {\bf 10} (2020), no.~3, Paper No. 39, 11 pp.

\bibitem{Ma2} W.~C. Ma, The Zalcman conjecture for close-to-convex functions, Proc. Amer. Math. Soc. {\bf 104} (1988), no.~3, 741--744.

\bibitem{Ma} W.~C. Ma, Generalized Zalcman conjecture for starlike and typically real functions, J. Math. Anal. Appl. {\bf 234} (1999), no.~1, 328--339.

\bibitem{ObrTun} M. Obradovi\'c{} and N. Tuneski, Hermitian Toeplitz determinants for the class $\mathcal S$ of univalent functions, Armen. J. Math. {\bf 13} (2021), Paper No. 4, 10 pp.

\bibitem{ProSzy} D.~V. Prokhorov and J. Szynal, Inverse coefficients for $(\alpha ,\beta )$-convex functions, Ann. Univ. Mariae Curie-Sk\l odowska Sect. A {\bf 35} (1981), 125--143 (1984).

\bibitem{RavVer} V. Ravichandran and S. Verma, Generalized Zalcman conjecture for some classes of analytic functions, J. Math. Anal. Appl. {\bf 450} (2017), no.~1, 592--605.

\bibitem{XuHeXu} Z. Xu, P. He and Q.~H. Xu, Some refinements of the Fekete and Szeg\"o{} inequalities and Toeplitz determinants in one and higher dimensions, Complex Anal. Oper. Theory {\bf 19} (2025), no.~3, Paper No. 62, 22 pp.

\bibitem{XuLiuLiu} Q.~H. Xu, T.~S. Liu and X.~S. Liu, Fekete and Szeg\"o{} problem in one and higher dimensions, Sci. China Math. {\bf 61} (2018), no.~10, 1775--1788.

\bibitem{XuJia} Q.~H. Xu and T. Jiang, The generalized Toeplitz determinants for a class of holomorphic mappings in several complex variables, Complex Anal. Oper. Theory {\bf 18} (2024), no.~6, Paper No. 140, 16 pp.

\end{thebibliography}
\end{document}